\documentclass[a4paper,12pt]{article}
\usepackage{amssymb,amsmath,amsthm,times,mathrsfs,fullpage}
\usepackage[pdftex]{hyperref}
\usepackage{tikz,graphicx}
\hypersetup{plainpages=True, pdfstartview=FitH, bookmarksopen=true,
pdfpagemode=none, colorlinks=true,linkcolor=blue,citecolor=blue}

\def\le{\leqslant}
\def\ge{\geqslant}
\def\dd#1{{\,\mathrm{d}}#1}
\def\ve{\varepsilon}
\def\tr#1{\left\lfloor #1\right\rfloor}
\def\tf{\tilde{f}}
\def\tV{\tilde{V}}
\def\JS{\mathscr{J\!\!S}}

\newtheorem{thm}{Theorem}
\newtheorem{lmm}{Lemma}
\newtheorem{cor}{Corollary}
\newtheorem{prop}{Proposition}
\newtheorem{definition}{Definition}

\title{A binomial splitting process
in connection with corner parking problems}
\author{Michael Fuchs\thanks{Partially supported by NSC under the grant NSC-102-2115-M-009-002}\\
    Department of Applied Mathematics\\
    National Chiao Tung University\\
    Hsinchu, 300\\ Taiwan
\and Hsien-Kuei Hwang\thanks{A significant part of the work of this
author was done while visiting ISM (Institute of Statistical
Mathematics), Tokyo; he thanks ISM for its hospitality and support.}
\\
    Institute of Statistical Science \\
    Institute of Information Science\\
    Academia Sinica\\
    Taipei 115\\
    Taiwan
\and Yoshiaki Itoh\thanks{This author thanks the Institute of
Statistical Science, Academia Sinica, Taipei, for its
hospitality and support.}\\
    Institute of Statistical Mathematics\\
    10-3 Midori-cho, Tachikawa\\
    Tokyo 190-8562\\
    Japan
\and Hosam H. Mahmoud\\
    Department of Statistics\\
    The George Washington University\\
    Washington, D.C. 20052\\
    U.S.A.}
\date{\today}

\begin{document}
\maketitle

\begin{abstract}
A special type of binomial splitting process is studied. Such a
process can be used to model a high-dimensional corner parking
problem, as well as the depth of random PATRICIA tries (a special
class of digital tree data structures). The latter also has natural
interpretations in terms of distinct values in iid geometric random
variables and the occupancy problem in urn models. The corresponding
distribution is marked by logarithmic mean and bounded variance,
which is oscillating, if the binomial parameter $p$ is not equal to
$1/2$, and asymptotic to 1 in the unbiased case. Also, the limiting
distribution does not exist owing to periodic fluctuations.
\end{abstract}


\medskip


\paragraph{Introduction.}
We study in this paper the random variables $X_n$ defined
recursively by
\begin{align} \label{Xn-rr}
    X_n \stackrel{d}{=} X_{I_n} + 1, \qquad
    \mbox{for \ } n \ge 1,
\end{align}
with $X_0=0$, where $(X_n)$ and $(I_n)$ are independent and
\[
    \mathbb{P}(I_n=k) = \binom{n}{k}
    \frac{p^kq^{n-k}-p^k(q-p)^{n-k}}
    {1-q^n}, \qquad \mbox{for } k=0,\dots,n-1,
\]
where, throughout this paper, $0<p\le q:=1-p$. In particular,
\begin{align*}
    p=\tfrac12 & \Rightarrow
    \mathbb{P}(I_n=k)=\binom{n}{k}\frac1{2^n-1},\\
    p=\tfrac13 & \Rightarrow \mathbb{P}(I_n=k)
    =\binom{n}{k}\frac{2^{n-k}-1}{3^n-2^n},
\end{align*}
for $k=0,\dots, n-1$. Note that, for convenience we retain the case
$k=0$, but drop $k=n$.

The random variables $X_n$ originally arose from the analysis of a
special type of parking problem with ``corner preference''
(described below). They would also arise in a leader election
algorithm that advances a truncated binomial number of contestants
at each stage. Namely, $X_n$ is the number of rounds till the
election comes to an end; see~\cite{Louchard,KM13} for a broad
framework for these types of problems.

On the other hand, it turns out that, when $p=1/2$, the distribution
of $X_n$ is identical to that of two parameters in a random
symmetric PATRICIA trie: the depth (distance between a uniformly
chosen leaf node and the root) and the length of the ``left arm''
(the path that starts at the root and keeps going left, until no
more nodes can be found on the left); see for example
\cite{KP87,RJS93}. Also, the depth or the left arm of random
PATRICIA tries is identically distributed as the number of distinct
values in some random sequences (see \cite{AKP06}), and the number
of occupied urns in some urn models (see \cite{HJ08}).

We prove in this paper that the random variables $X_n$ have
logarithmic mean and bounded variance for large $n$. Also the
distributions do not approach a fixed limit law due to the inherent
fluctuations. For a similar context, see \cite{JLL08,JS97} and the
references therein; see also the recent paper \cite{KM13}.

\paragraph{A corner preference parking problem in discrete space.}
The parking problem has a long history in the discrete probability
literature, and is closely connected to many applications and models
in chemistry, physics, biology and computer algorithms; see
\cite{SI11,Evans93}. Most analytic results for the numerous variants
in the literature have to do with one-dimensional settings, and very
few deal with higher dimensions due to the intrinsic complexity of
the corresponding equations.

We first explain a simple discrete parking problem. Integral
translates of the cube $[0,\ell]^n$ are ``parked" into the
$n$-dimensional hypercube $[0,L]^n$, where $L>\ell\ge 1$. A precise
mathematical formulation of this is as follows: represent cubes by
their corner which has the shortest distance to the origin.
Moreover, set
\[
    Z_{L-\ell}^n :=\{\mathbf{a}=(a_1,\dots,a_n)
    \,:\, a_{j}=0,1,...,L-\ell ,
    \text{for } 1\le j \le n \},
\]
and define a distance $\rho(\mathbf{x}, \mathbf{y}) :=\max_{1 \le j
\le n} |x_j-y_j|$ between two points $\mathbf{x}, \mathbf{y} \in
Z_{L-\ell}^n$. At first, choose one point uniformly at random from
$Z_{L-\ell}^n$ and record it as $\mathbf{a}(1)$. Then choose another
point uniformly at random and record it as $\mathbf{a}(2)$ if $\ell
\le \rho(\mathbf{a}(1), \mathbf{a}(2))$; otherwise, reject it and
repeat the same procedure. If $\mathbf{a}(1), \mathbf{a}(2), \dots,
\mathbf{a}(k)$ are already recorded, choose the next point uniformly
at random and record it as $\mathbf{a}(k+1)$ if $\ell \le
\rho(\mathbf{a}(k+1), \mathbf{a}(j))$ for $j=1,2,\dots, k$;
otherwise, reject it and repeat the same procedure. We continue this
procedure until it is impossible to add more points among the
$(L-\ell +1)^{n}$ points. Since analytic development of this model
remains challenging, simulations have been carried out for finding
the jamming density of this model; see \cite{IS86,IU83}.

We further restrict the parking to be operated along one direction
only, which we call ``corner preference parking.'' More precisely,
let
\begin{align*}
    S(\mathbf{a})
    &:=\{ \mathbf{x} \,:\, \mathbf{x} \in Z_{L-\ell}^n
    , 0\le x_j \le a_j, j=1,\dots,n\},\\
    U(\mathbf{a}, \ell) &: =\{\mathbf{x} \,:\, \mathbf{x}
    \in Z_{L-\ell}^n, \rho(\mathbf{x},  \mathbf{a}) < \ell \},
\end{align*}
and $S_U (\mathbf{a}, \ell):=S(\mathbf{a})\setminus U(\mathbf{a},
\ell)$.

The corner preference parking problem then starts from
$S_U(\mathbf{a(1)}, \ell)$ with $\mathbf{a(1)}= (L-\ell,
L-\ell,...,L-\ell)\in Z_{L-\ell}^n$. After that, we place
sequentially at random the integral translates of cube $[0,\ell]^n$
into the cube $[0,L]^n$,  so that any car  placed  is closer to the
``corner" (origin) than the previously placed cubes, until  there is
no possible space to park. By ``closer to the corner'' we mean that
the coordinates of the point representing the car are all at most
as large as the car parked immediately before it. The process
continues till saturation.

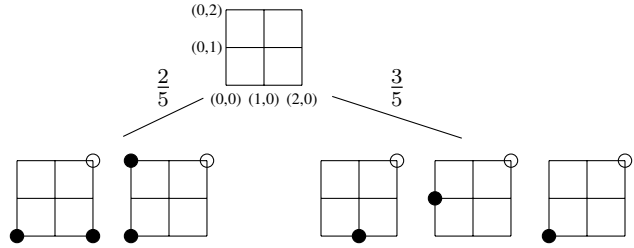
\begin{figure}[!ht]
\begin{center}
\begin{tikzpicture}[scale=1]
\node (a) at (-3.4,1.9) {};
\node (b) at (-4.75,1.25) {};
\draw (a) -- node[above] {$\frac{2}{5}$} (b);
\node (c) at (-2,1.9) {};
\node (d) at (0,1.25) {};
\draw (c) -- node[above] {$\frac{3}{5}$} (d);
\draw[fill] (-6,0) circle [radius=2.5pt] ;
\draw[fill] (-5,0) circle [radius=2.5pt] ;
\draw (-5,1) circle [radius=2.5pt] ;
\draw[fill] (-4.5,0) circle [radius=2.5pt] ;
\draw[fill] (-4.5,1) circle [radius=2.5pt] ;
\draw (-3.5,1) circle [radius=2.5pt] ;
\draw[fill] (-1.5,0) circle [radius=2.5pt] ;
\draw (-1,1) circle [radius=2.5pt] ;
\draw[fill] (-0.5,0.5) circle [radius=2.5pt] ;
\draw (0.5,1) circle [radius=2.5pt] ;
\draw[fill] (1,0) circle [radius=2.5pt] ;
\draw (2,1) circle [radius=2.5pt] ;

\draw[step=.5,black,thin,xshift=-0.25cm]
(-3,2-0.001) grid (-2,3); 
\node at (-3.25,1.8) {\tiny(0,0)};
\node at (-2.75,1.8) {\tiny(1,0)};
\node at (-2.25,1.8) {\tiny(2,0)};
\node at (-3.5,2.49) {\tiny(0,1)};
\node at (-3.5,2.99) {\tiny(0,2)};
\draw[step=.5,black,thin] (-6,0) grid (-5,1); 
\draw[step=.5,black,thin] (-4.5,0) grid (-3.5,1);
\draw[step=.5,black,thin] (-2,0) grid (-1,1);
\draw[step=0.5,black,thin] (-0.5,0) grid (0.5,1);
\draw[step=0.5,black,thin]
(1-0.001,0-0.001) grid (2,1);
\end{tikzpicture}
\end{center}
\caption{\emph{The five ($3^2-2^2$) different two-dimensional
configurations of corner preference parking, when $L=4$ and
$\ell=2$. In this case $\mathbb{E}(u^{X_2})= \frac35u+\frac25u^2$.}}
\label{fig-2d}
\end{figure}

Take now $L=2m$, and $\ell=m$, where $m\ge1$. Assume that all
possible ``parking positions" are equally likely at each stage. Let
the random variable $Y_n$ be the number of cars parked after the
first car at the time of saturation in such an $n$--dimensional
corner parking problem. The distribution of $Y_n$ can be explicitly
characterized.
\begin{lmm} The random variables $Y_n$ can be recursively enumerated
by
\begin{equation}\label{lcp-rr}
    \mathbb{E}(u^{Y_n}) = u \sum_{1\le k\le n}
    \binom{n}{k}\frac{m^k-(m-1)^k}{(m+1)^n-m^n}\,
    \mathbb{E}\bigl(u^{Y_{n-k}}\bigr),
    \qquad \mbox{for \ } n\ge1,
\end{equation}
with $Y_0=0$.
\end{lmm}
This corresponds to \eqref{Xn-rr} with $p=1/(m+1)$.
\begin{proof}
There are $(m+1)^n$ possible positions of integral translates of
hypercubes (cars) $[0,m]^n$ to park in the hypercube $[0,2m]^n$.
After parking the first car at the top right corner, the number of
possible positions for the second car $\mathbf{a}(2)$ equals
$(m+1)^n-m^n$, which is the denominator in \eqref{lcp-rr}.

Suppose that $n-k$ coordinates of $\mathbf{a}(2)$ assume the value
$m$, and each of the remaining $k$ coordinates may assume any of
the values $\{0,1,\ldots,m-1\}$. Observe that at least one of the
$k$ coordinates should be $0$ for the second car to park without
overlapping with the first car. Since the number of cases with each
of the $k$ coordinates taking a value in $\{1, 2,\ldots , m-1\}$
is $(m-1)^k$, the number of all possible positions for the second
car (under the mentioned restriction) is $m^k-(m-1)^{k}$. We have
${n \choose k}$ choices of the $k$ coordinates. Thus,
under this restriction, the second car has a total of
${n \choose k}(m^k-(m-1)^{k})$ possible positions to park. After
this placement, the problem is reduced to that of an
$(n-k)$-dimensional one; see Figure~\ref{fig-2d} for an illustration
with $L=4$ and $\ell=2$.
\end{proof}

\paragraph{Depths of PATRICIA tries.} Tries (a mixture of tree and
re\emph{trie}val) are one of the most useful tree structures in
storing alphabetical or digital data in computer algorithms, the
underlying construction principle of which being simply ``0-bit
directing to the left" and ``1-bit directing to the right".
PATRICIA\footnote{``PATRICIA" is an acronym, which stands for
``Practical Algorithm To Retrieve Information Coded In
Alphanumeric".} tries are a variant of tries where all nodes with
only a single child are compressed; see Figure~\ref{fg-trie} for a
plot of tries and PATRICIA tries and the book \cite{Mahmoud92} for
more information. Note that, unlike tries whose number of internal
nodes is not necessarily a constant, a PATRICIA trie of $n$ keys has
always $n-1$ internal nodes for branching purposes, a standard
property of a tree.

To study the shapes of random PATRICIA tries, we assume that the
input is a sequence of $n$ independent and identically distributed
random variables, each composed of an infinite sequence of Bernoulli
random variables with mean $p$, $0<p<1$. Under such a Bernoulli
model, we construct random PATRICIA tries and the shape parameters
become random variables.

Consider the depth $Z_n$ of a random PATRICIA trie of $n$ keys under
the Bernoulli model, where the depth denotes the distance between
the root and a randomly chosen key (from the leaves where keys are
stored), where the $n$ keys are equally likely to be selected. Then
we have the recurrence relation for the probability generating
function of $Z_n$
\[
    \mathbb{E}(u^{Z_n}) = u\sum_{1\le k<n}
    \frac{\binom{n}{k}p^k q^{n-k}}{1-p^n-q^n}
    \left(\frac{k}{n}\,\mathbb{E}(u^{Z_k})
    + \frac{n-k}{n}\,\mathbb{E}(u^{Z_{n-k}})
    \right), \qquad \mbox{for \ }n\ge2,
\]
with $Z_0=Z_1=0$. In the unbiased case $p=1/2$, this reduces to
\[
    \mathbb{E}(u^{Z_n}) = u\sum_{0\le k\le n-2}
    \frac{\binom{n-1}{k}}{2^{n-1}-1}
    \,\mathbb{E}(u^{Z_{k+1}}),
\]
which implies that
\[
    Z_{n+1} \stackrel{d}{\equiv} X_n, \qquad
    \mbox{for\ } n\ge1; \quad p=\frac 1 2.
\]

\begin{figure}
\begin{center}
\begin{tikzpicture}[line width=1pt,scale=0.65]
\tikzstyle{every node}=[draw ,circle, level distance=0.2cm]
\tikzstyle{level 1}=[sibling distance=7cm]
\tikzstyle{level 2}=[sibling distance=3cm]
\node {}
    child{ node{}
        child{node[rectangle]{\small $00$}
        edge from parent node[draw=none, left] {\small $0$}}
        child{node{}
            child{ edge from parent [draw=none]}
            child{node{}
                child{node{}
                    child{node[rectangle]{\small $01100$}
                    edge from parent node[draw=none, left] {\small $0$}}
                    child{node[rectangle]{\small $01101$}
                    edge from parent node[draw=none, right] {\small $1$}}
                edge from parent node[draw=none, left] {\small $0$}}
                child{edge from parent[draw=none]}
            edge from parent node[draw=none, right] {\small $1$}}
        edge from parent node[draw=none, right] {\small $1$}}
    edge from parent node[draw=none, left] {\small $0$}}
    child{ node{}
        child{node{}
            child{
            edge from parent[draw=none]}
            child{node{}
                child{node[rectangle]{\small $1010$}
                edge from parent node[draw=none, left] {\small $0$}}
                child{node[rectangle]{\small $1011$}
                edge from parent node[draw=none, right] {\small $1$}}
            edge from parent node[draw=none, right] {\small $1$}}
        edge from parent node[draw=none, left] {\small $0$}}
        child{
        edge from parent[draw=none]}
    edge from parent node[draw=none, right] {$1$}};
\end{tikzpicture}
\begin{tikzpicture}[line width=1pt,scale=0.65,baseline=-4.5cm]
\tikzstyle{every node}=[draw ,circle, level distance=0.2cm]
\tikzstyle{level 1}=[sibling distance=7cm]
\tikzstyle{level 2}=[sibling distance=3cm]
\node {}
    child{ node{}
        child{node[rectangle]{$00$}
        edge from parent node[draw=none, left] {\small $0$}}
        child{node{\small $10$}
            child{node[rectangle]{\small $01100$}
            edge from parent node[draw=none, left] {\small $0$}}
            child{node[rectangle]{\small $01101$}
            edge from parent node[draw=none, right] {\small $1$}}
        edge from parent node[draw=none, right] {\small $1$}}
    edge from parent node[draw=none, left] {\small $0$}}
    child{node{\small $01$}
        child{node[rectangle]{\small $1010$}
        edge from parent node[draw=none, left] {\small $0$}}
        child{node[rectangle]{\small $1011$}
        edge from parent node[draw=none, right] {\small $1$}}
    edge from parent node[draw=none, right] {\small $1$}};
\end{tikzpicture}
\caption{\emph{A trie (left) of $n=5$ records, and the corresponding
PATRICIA tries (right): the circles represent internal nodes and
rectangles holding the records are external nodes. The compressed
bits are also indicated on the nodes.}} \label{fg-trie}
\end{center}
\end{figure}
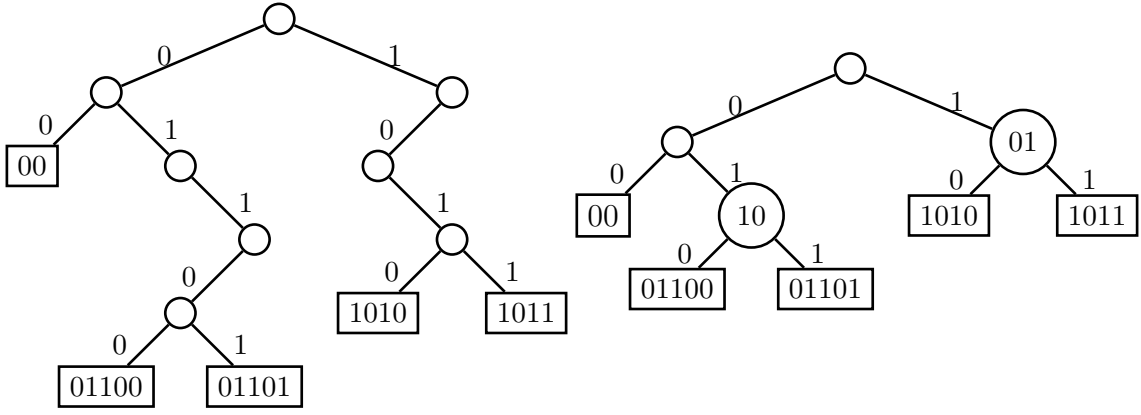

Another identically distributed random variable is the length $W_n$
of the ``left arm", which is the path starting from the root and
going always to the left until reaching a key-node. Then, under the
Bernoulli model,
\begin{align}\label{Wn}
    \mathbb{E}(u^{W_n}) = u\sum_{1\le k\le n}
    \frac{\binom{n}{k}p^k q^{n-k}}{1-q^n}\,
    \mathbb{E}(u^{W_{n-k}}),
    \qquad \mbox{for\ } n\ge2,
\end{align}
with $W_0=0$ and $W_1=1$. It is obvious that $W_n
\stackrel{d}{\equiv} X_n$, when $p=1/2$.

\paragraph{Distinct values and urn models.} The left arm $W_n$ has
yet two other different interpretations. One is in terms of the
number of distinct letters in a sequence of independent and
identically distributed geometric random variables with success
probability $p$ for which one has exactly the recurrence \eqref{Wn};
see \cite{AKP06}. Alternatively, if we consider the urn model where
the $j$th urn has probability of $pq^j$ of receiving a ball, then
the number of occupied urns also follows the same distribution; see
\cite{HJ08}.

\paragraph{Exponential generating functions.} The exponential
generating function for the moment generating function of $X_n$,
\[
    P(z,y) := \sum_{n\ge0} \frac{\mathbb{E}(e^{X_n y})}{n!}\, z^n,
\]
satisfies, by \eqref{Xn-rr}, the functional equation
\[
    P(z,y) = e^y (e^{qz}-e^{(q-p)z})P(pz,y)+P(qz,y).
\]
It follows that the exponential generating function of the mean
$f_1(z) := \sum_{n\ge0} \mathbb{E} (X_n)z^n/n!$ satisfies
\[
    f_1(z) = (e^{qz}-e^{(q-p)z})f_1(pz)+f_1(qz)+e^z-e^{qz},
\]
with $f_1(0)=0$. By iteration, we obtain
\begin{align*}
    f_1(z) &= \sum_{\substack{k\ge0\\ 0\le j\le k}}
    \left(e^{p^{k-j}q^jz}-e^{p^{k-j}q^{j+1}z}\right)\\
    &\qquad\qquad\times\sum_{0\le i_1\le\cdots\le i_{k-j}\le j}
    \prod_{0\le \ell <k-j}  \left(e^{p^{\ell}q^{i_{\ell+1}+1}z}
    -e^{(q-p)p^{\ell}q^{i_{\ell+1}}z}\right) ,
\end{align*}
which does not seem useful for further manipulation.

\paragraph{Poisson generating functions.} For our asymptotic
purposes, it is technically more convenient to consider the Poisson
generating function,
\[
    \tilde{P}(z,y) := e^{-z}P(z,y),
\]
which then satisfies the equation
\begin{equation}\label{bivar-gf}
    \tilde{P}(z,y) =
    e^y (1-e^{-pz})\tilde{P}(pz,y)+e^{-pz}\tilde{P}(qz,y).
\end{equation}

It follows that the Poisson generating function for the $m$th moment
\[
    \tilde{f}_m(z)
    := e^{-z}\sum_{n\ge0}\frac{\mathbb{E}(X_n^m)}{n!}\, z^n,
\]
satisfies the equation
\[
    \tilde{f}_m(z) =
    (1-e^{-pz})\sum_{0\le \ell \le m}
    \binom{m}{\ell} \tilde{f}_\ell(pz)
    +e^{-pz}\tilde{f}_m(qz), \qquad \mbox{for \ } m\ge0,
\]
where $\tilde{f}_0(z) = 1$.

In particular, we have
\begin{align}
    \tilde{f}_1(z) &=
    (1-e^{-pz})\tilde{f}_1(pz)+e^{-pz}\tilde{f}_1(qz)
    + 1-e^{-pz}, \label{tM1z}\\
    \tilde{f}_2(z) &=
    (1-e^{-pz})\tilde{f}_2(pz)+ e^{-pz}\tilde{f}_2(qz)
    + 2(1-e^{-pz})\tilde{f}_1(pz)+
    1-e^{-pz}. \label{tM2z}
\end{align}

\paragraph{Expected value of $X_n$.}
Let
\begin{align}\label{phi-z}
    \phi(z) = e^{-z}\left(\tilde{f}_1\left(p^{-1} qz\right)
    -\tilde{f}_1(z)\right),
\end{align}
and let $\phi^*(s)$ denote its Mellin transform (see \cite{FGD95})
\begin{align} \label{phi-s}
    \phi^*(s) := \int_0^\infty
    e^{-t}t^{s-1} \left(\tilde{f}_1\left(p^{-1} qt\right)
    -\tilde{f}_1(t)\right) \dd t,
\end{align}
which is well-defined in the half-plane $\Re(s)>-1$ (see Appendix
for growth properties of $\tilde{f}_1$).

\begin{thm} The expected value of $X_n$ satisfies
\label{thm:mun}
\begin{align*}
    \mathbb{E}(X_n)
    &= \log_{1/p} n +\frac{\gamma+\phi^*(0)}{\log(1/p)}-
    \frac12 + Q(\log_{1/p} n) + O(n^{-1}).
\end{align*}
Here $\gamma$ denotes Euler's constant and
\begin{align} \label{C:Qu}
    Q(u) := \sum_{k\in\mathbb{Z}\setminus\{0\}}
    Q_k e^{-2k\pi i u}, \qquad
    Q_k := -\frac{\Gamma(\chi_k)-
    \phi^*(\chi_k)}{\log(1/p)},
\end{align}
where $\chi_k := 2k\pi i/\log(1/p)$, and $\Gamma$ denotes the
Gamma function.
\end{thm}

The asymptotic expansion simplifies when $p=1/2$; indeed, in this
case, we have the closed-form expression
\[
    \tilde{f}_1(z) = \sum_{k\ge1}\left(1-e^{-z/2^k}\right),
    \qquad\mbox{for \ } \Re(z)>0.
\]
\begin{cor} In the symmetric case when $p=1/2$, the expected value
of $X_n$ satisfies asymptotically
\[
    \mathbb{E}(X_n) =
    \log_2n +\frac{\gamma}{\log 2}-\frac12 -\frac1{\log 2}
    \sum_{k\ne0} \Gamma(\chi_k) n^{-\chi_k} + O\left(n^{-1}\right),
\]
where $\chi_k = 2k\pi i/\log 2$.
\end{cor}

For numerical purposes, the value of $\phi^*(\chi_k)$ can be
computed by the series expression
\[
    \phi^*(\chi_k) = \sum_{j\ge1}\frac{\mu_j}{j!}\,
    \Gamma(\chi_k+j) \left(q^j - 2^{-j-\chi_k}\right),\qquad
    \mbox{for \ } k=0,1,\dots \ .
\]

Approximate plots of the periodic function $Q(u)$ for $p=1/3$ based
on exact values of $\mu_n:={\mathbb E}(X_n)$ and on its Fourier series
are given in Figure~\ref{C:fig}.
\begin{figure}
\begin{center}
\includegraphics[width=7cm]{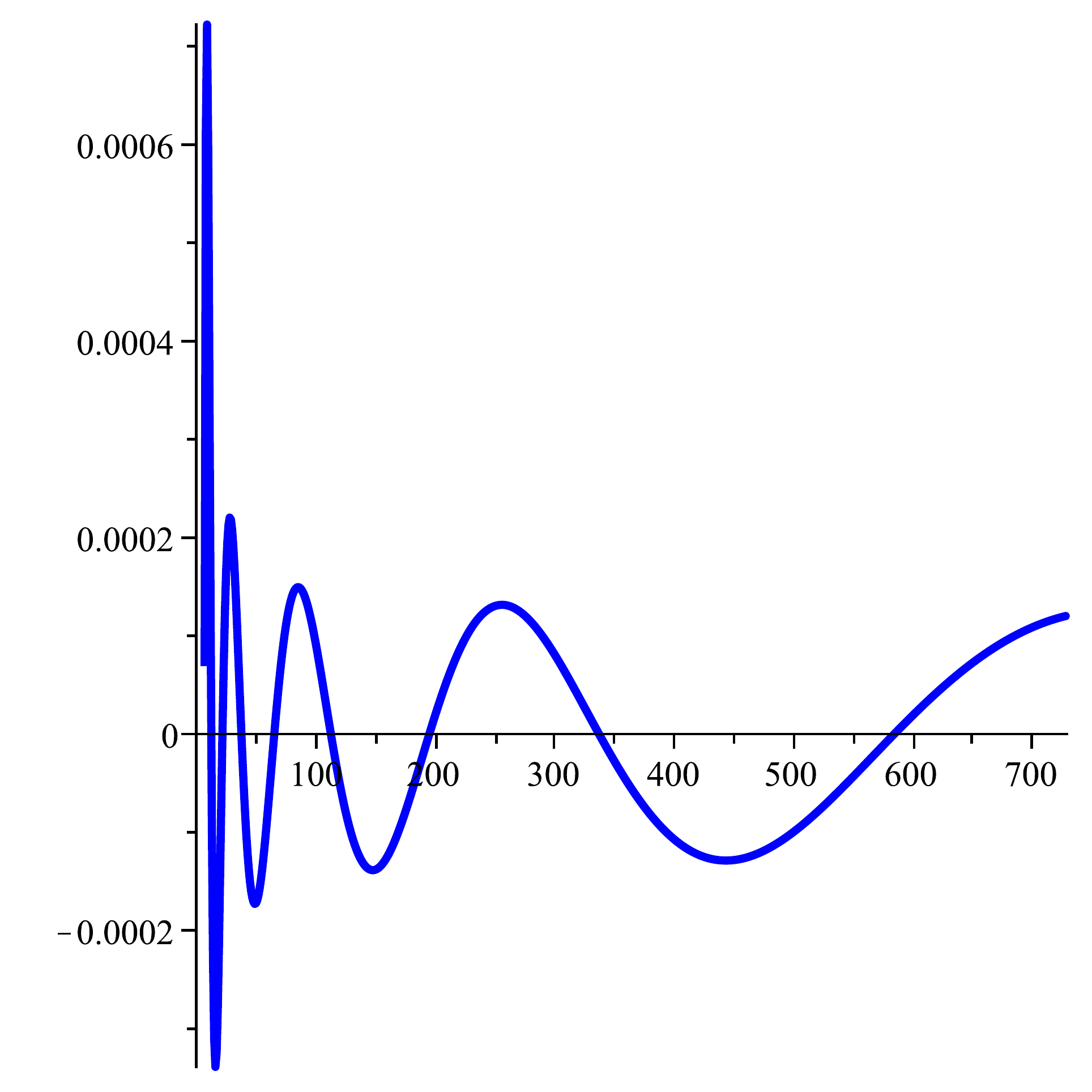}
\includegraphics[width=6.5cm]{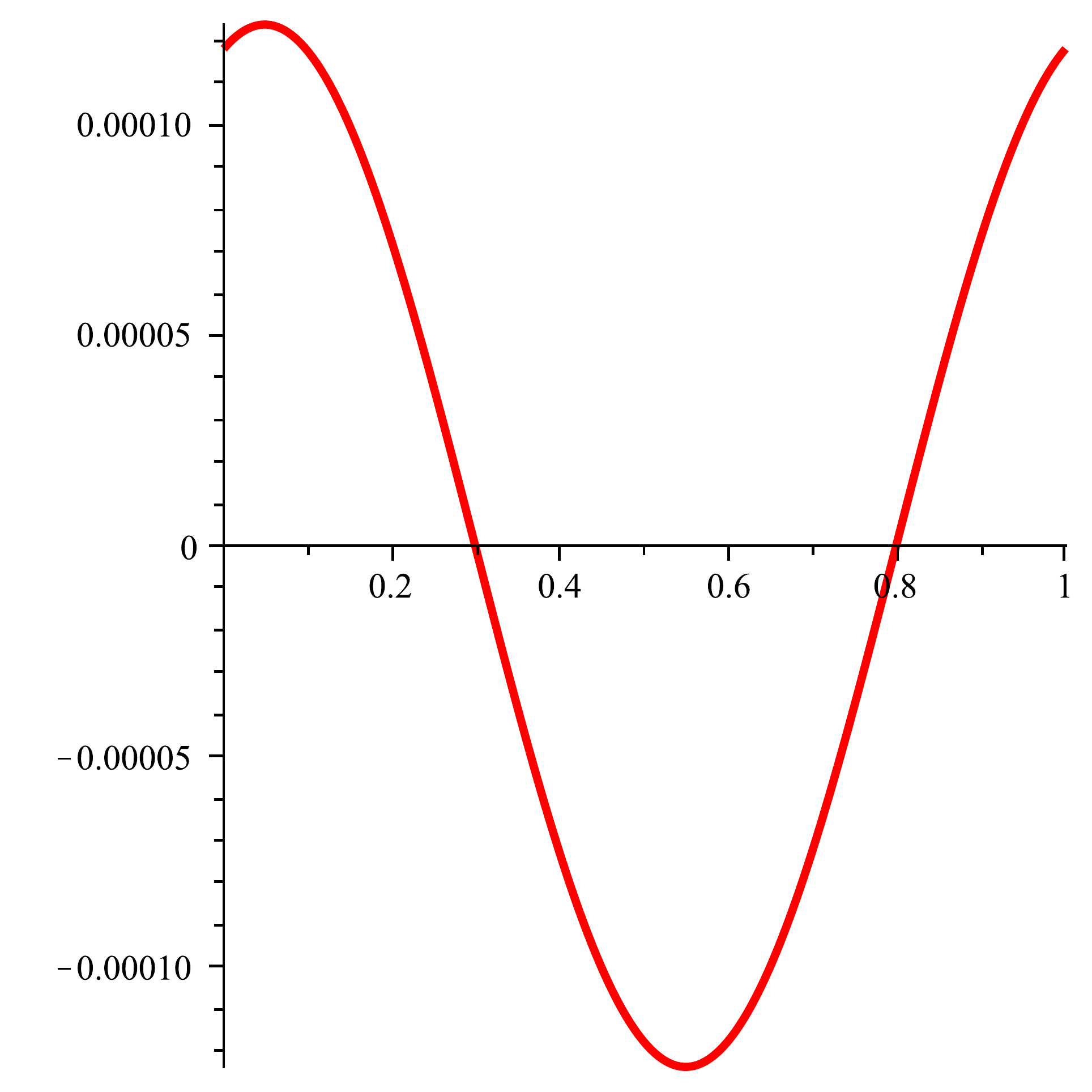}
\end{center}
\caption{\em $p=1/3$: Fluctuation of the periodic function $Q(\log_3
n)$, as approximated by $\mu_n - H_n/\log 3+1/2-\phi^*(0)/\log 3$
(left) and the first five terms of the Fourier series \eqref{C:Qu}
(right).} \label{C:fig}
\end{figure}

\paragraph{Outline of proof.} Theorem~\ref{thm:mun} is proved by
a two-stage, purely analytic approach based on Mellin transform and
analytic de-Poissonization (see \cite{FS09,JS98}). We outline the
major steps and arguments used here, leaving the major technical
justification in the Appendix.

Our starting point is the functional equation~\eqref{tM1z}, which is
rewritten as
\[
    \tilde{f}_1(z) = \tilde{f}_1(pz) + \phi(pz) + 1-e^{-pz},
\]
where $\phi(z)$ is defined in \eqref{phi-z}. While $\phi$ involves
itself $\tilde{f}$, we show that it is exponentially small for large
complex parameter, and thus the asymptotics of $\tilde{f}_1(z)$ can
be readily derived by standard inverse Mellin transform arguments
(growth order of the integrand at infinity and calculus of
residues).

Once the asymptotics of $\tilde{f}_1(z)$ for large $|z|$ is known,
we can apply the Cauchy integral formula
\[
    \mathbb{E}(X_n) = \frac{n!}{2\pi i}
    \oint_{|z|=n} z^{-n-1} e^z \tilde{f}_1(z) \dd z,
\]
and the saddle-point method to derive the asymptotics of the mean.
Roughly, the growth order of $\tilde{f}_1$ is small, meaning that
the saddle-point (where the derivative of the integrand becomes
zero) lies near $n$. The specialization of the saddle-point method
here (with integration contour $|z|=n$) has many interesting
properties and is often referred to as the analytic
de-Poissonization (see the survey paper by Jacquet and Szpankowski
\cite{JS98}).

It turns out that such a Mellin and de-Poissonization process can be
manipulated in a rather systematic and operational manner by
introducing the notion of JS-admissible functions in which we
combine ideas from \cite{JS98} and \cite{Hayman56} (see also
\cite{FHV13,HFV10}). So we can easily apply the same approach to
characterize the asymptotics of the variance and the limiting
distribution.

\paragraph{Mellin transform.} Let
\begin{align} \label{scrS}
    \mathscr{S}_\ve := \{z\, :\,  |\arg(z)|\le \pi/2-\ve\},
    \qquad \mbox{for \ } \ve>0.
\end{align}
By Proposition~\ref{prop-tr} (in Appendix), $\tilde{f}_1(z)$ is
polynomially bounded for large $|z|$ in the sector
$\mathscr{S}_\ve$. This means that $\phi(z) + 1-e^{-pz} =O(1)$ for
$|z|\ge 1$ in $\mathscr{S}_\ve$. Consequently, $\tilde{f}_1(z) =
O(|\log z|)$ in the same range of $z$. On the other hand, since
$\tf_1(z) \sim z$, as $z\to0$, we see that the Mellin transform
\[
    \tf_1^\star(s) := \int_0^\infty \tf_1(z) z^{s-1} \dd z,
\]
exists in the strip $-1<\Re(s)<0$, and defines an analytic function
there.

It follows from \eqref{tM1z} that
\[
    \tf_1^\star(s) =
    \frac{\Gamma(s)-\phi^*(s)}{1-p^s},
    \qquad \mbox{for \ } -1<\Re(s)<0,
\]
and $\phi^*$ is defined in \eqref{phi-s}.

By the Mellin inversion formula,
\[
    \tf_1(z) = \frac1{2\pi i}
    \int_{-1/2-i\infty}^{-1/2+i\infty}
    \frac{\Gamma(s)-\phi^*(s)}{1-p^s}\,z^{-s} \dd s.
\]
We need the growth property of $\phi^*(\sigma\pm it)$ for
large $|t|$.

\begin{lmm} For $\sigma>-1$,
\[
    |\phi^*(\sigma\pm it)| = O(e^{-(\pi/2-\ve)|t|}),
\]
as $|t|\to\infty$.
\end{lmm}
\begin{proof}
This follows from the fact that $\tf_1(z)$ is an entire function,
the estimate $\tf_1(z)=O(|\log z|)$ for $z\in \mathscr{S}_\ve$ and
the Exponential Smallness Lemma (\cite[Proposition 5]{FGD95}).
\end{proof}

On the other hand, since
\[
    |\Gamma(\sigma\pm it)| = O\left(|t|^{\sigma-1/2}
    e^{-\pi|t|/2}\right),
\]
for finite $\sigma$ and $|t|\to\infty$, we can move the line of
integration to the right, summing the residues of all poles
encountered. The result is
\begin{equation}
    \tf_1(z) = \log_{1/p} z +C
    +Q(\log_{1/p} z) +\frac1{2\pi i}
    \int_{1/2-i\infty}^{1/2+i\infty}
    \frac{\Gamma(s)-\phi^*(s)}{1-p^s}\,z^{-s} \dd s,
    \label{C:F-M}
\end{equation}
where (defining $\chi_k:= 2k\pi i/\log(1/p)$)
\begin{align} \label{C:C}
    C := -\frac12 +\frac{\gamma+\phi^*(0)}{\log(1/p)}.
\end{align}
Note that, by definition, we have
\[
    \phi^*(\chi_k)= \sum_{j\ge1} \frac{\mu_j}{j!}\,
    \Gamma(j+\chi_k)
    \left(q^j-2^{-j-\chi_k}\right),
    \qquad \mbox{for \ } k\in\mathbb{Z},
\]
the series being absolutely convergent by the growth order of
$\mu_j$. In particular, when $p=1/3$
\begin{align*}
    \phi^*(0) = \sum_{j\ge1} \frac{\mu_j}{j}\left( \frac{2^j}{3^j}
    -\frac1{2^j}\right) \approx 0.58130\,98083\,52813\,44019\dots,
\end{align*}
so that $C\approx 0.55453\,53308\,02526\,96605\dots$.

To evaluate the remainder integral in \eqref{C:F-M}, we expand the
factor $1/(1-p^s)$ (since $\Re(s)>0$) into a geometric series, and
integrate term by term, giving
\[
    \frac1{2\pi i}
    \int_{1/2-i\infty}^{1/2+i\infty}
    \frac{\Gamma(s)-\phi^*(s)}{1-p^s}\,z^{-s} \dd s =
    \sum_{k\ge0} e^{-p^{-k}z} \left(1-\tf_1(qp^{-k-1}z)
    +\tf_1(p^{-k}z)\right).
\]
Thus the remainder is indeed exponentially small.

We summarize these derivations as follows.
\begin{prop} For $z$ lying in the sector $\mathscr{S}_\ve$,
$\tf_1(z)$ satisfies the asymptotic and exact formula
\begin{equation}\label{C:Fz-exact}
    \tf_1(z) = \log_{1/p} z +C +Q(\log_{1/p} z) +\sum_{k\ge0}
    e^{-p^{-k}z} \left(1-\tf_1(qp^{-k-1}z)+\tf_1(p^{-k}z)\right),
\end{equation}
where $C$ and $Q$ are given in \eqref{C:C} and \eqref{C:Qu},
respectively.
\end{prop}

Theorem~\ref{thm:mun} then follows from standard de-Poissonization
argument (see Appendix)
\[
    \mathbb{E}(X_n) = \tilde{f}_1(n)- \frac n2\tilde{f}_1''(n)
    + O\left(n^{-2}\right),
\]
and \eqref{C:Fz-exact}.

\paragraph{The variance.} For the asymptotics of the variance, it
proves advantageous to consider suitable Poissonized variance at the
generating function level, which, in the case of $X_n$, can be
handled by the following form
\[
    \tV(z) := \tf_2(z) - \tf_1(z)^2,
\]
where the Poisson generating function of the second moment
$\tf_2(z)$ satisfies the equation \eqref{tM2z}. Then $\tV(z)$
satisfies the functional equation
\begin{align}\label{tVz}
    \tV(z) =  (1-e^{-pz})\tV(pz)+
    e^{-pz}\tV(qz) + g_V^{}(z),
\end{align}
with $\tV(0)=0$, where
\[
    g_V^{}(z) := e^{-pz}\left(1-e^{-pz}\right)
    \left(1+\tf_1(pz) - \tf_1(qz)\right)^2.
\]
Unlike $g$ and $g_2$, which is $O(1)$ for large $z$, $g_V^{}$
is exponentially small for large $z$.

When $p=1/2$, we see that \eqref{tVz} has the closed-form solution
\[
    \tilde{V}(z) = 1-e^{-z}.
\]

When $p\ne q$, define
\[
    \phi_V^*(s) := \int_0^\infty
    e^{-z} z^{s-1}\left(\tV\left(p^{-1} qz\right)-\tV(z)
    +\left(1-e^{-z}\right)
    \left(1+\tf_1(z)
    -\tf_1\left(p^{-1} qz\right)\right)^2\right) \dd z.
\]
By following exactly the same analysis as that for $\tf_1$, we
obtain
\begin{align}\label{tVz-id}
\begin{split}
    \tV(z) &= Q_V(\log_{1/p}z) + \sum_{k\ge0}
    e^{-p^{-k}z} \biggl\{\tV\left(p^{-k-1} qz\right)-\tV(p^{-k}z)
    \\ &\hspace*{3cm}
    +\left(1-e^{-p^{-k}z}\right)
    \left(1+\tf_1(p^{-k}z)
    -\tf_1\left(p^{-k-1} qz\right)\right)^2\biggr\},
\end{split}
\end{align}
for $\Re(z)>0$, where
\[
    Q_V(u) = \frac1{\log(1/p)}\sum_{k\in\mathbb{Z}}
    \phi_V^*(\chi_k) e^{-2k\pi i u}.
\]
Note that $Q_V(u)=1$ when $p=1/2$.
\begin{thm} If $p=1/2$, then the variance of $X_n$ satisfies
\[
    \mathbb{V}(X_n) = 1 + O(n^{-1});
\]
if $p\ne q$, then the variance of $X_n$ is bounded and
asymptotically periodic in nature
\[
    \mathbb{V}(X_n) = Q_V(\log_{1/p}n) + O(n^{-1}).
\]
\end{thm}
\begin{proof} By the definition of $\tilde{V}$
\begin{align*}
    \mathbb{V}(X_n) &= n![z^n]e^z\tilde{f}_2(z) -
    \left(\mathbb{E}(X_n)\right)^2\\
    &= n! [z^n]e^z \tilde{V}(z) - n! [z^n]e^z \tilde{f}_1(z)^2
    - \left(n![z^n]e^z\tilde{f}_1(z)\right)^2,
\end{align*}
which, by the asymptotic nature of the Poisson-Charlier expansions
(see Appendix), is asymptotic to
\begin{align*}
    \mathbb{V}(X_n) &= \tilde{V}(n) +O\left(n\tilde{V}''(n)
    + n\tilde{f}_1'(n)^2\right)\\
    &= \tilde{V}(n) + O\left(n^{-1}\right),
\end{align*}
and the theorem follows from \eqref{tVz-id}.
\end{proof}

Figure~\ref{C:fig2} illustrates the periodic fluctuations of the
variance when $p=1/3$. See also \cite{Prodinger04} for a similar
situation where the variance is not oscillating when $p=1/2$.
\begin{figure}
\begin{center}
\includegraphics[width=7cm]{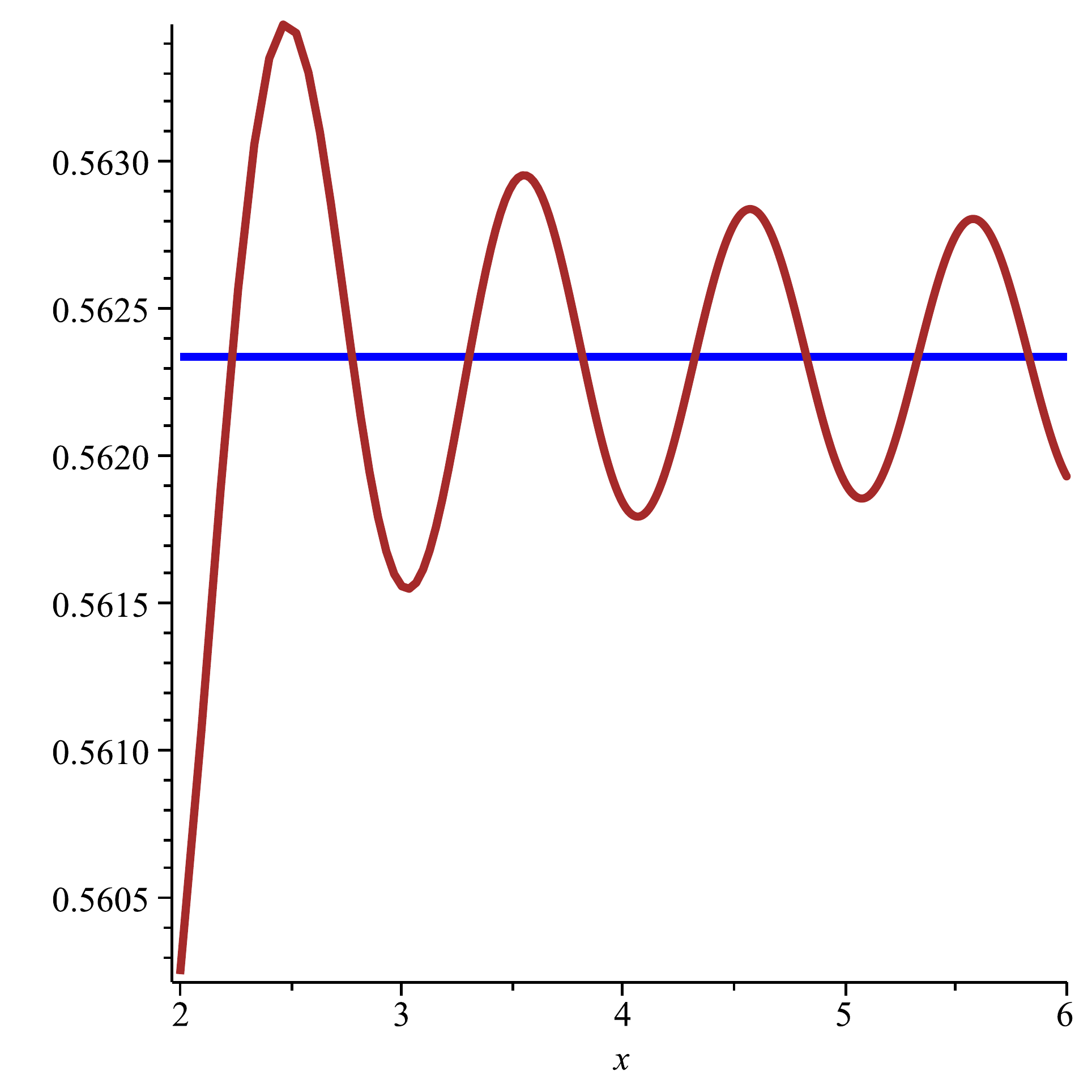}
\includegraphics[width=6.5cm]{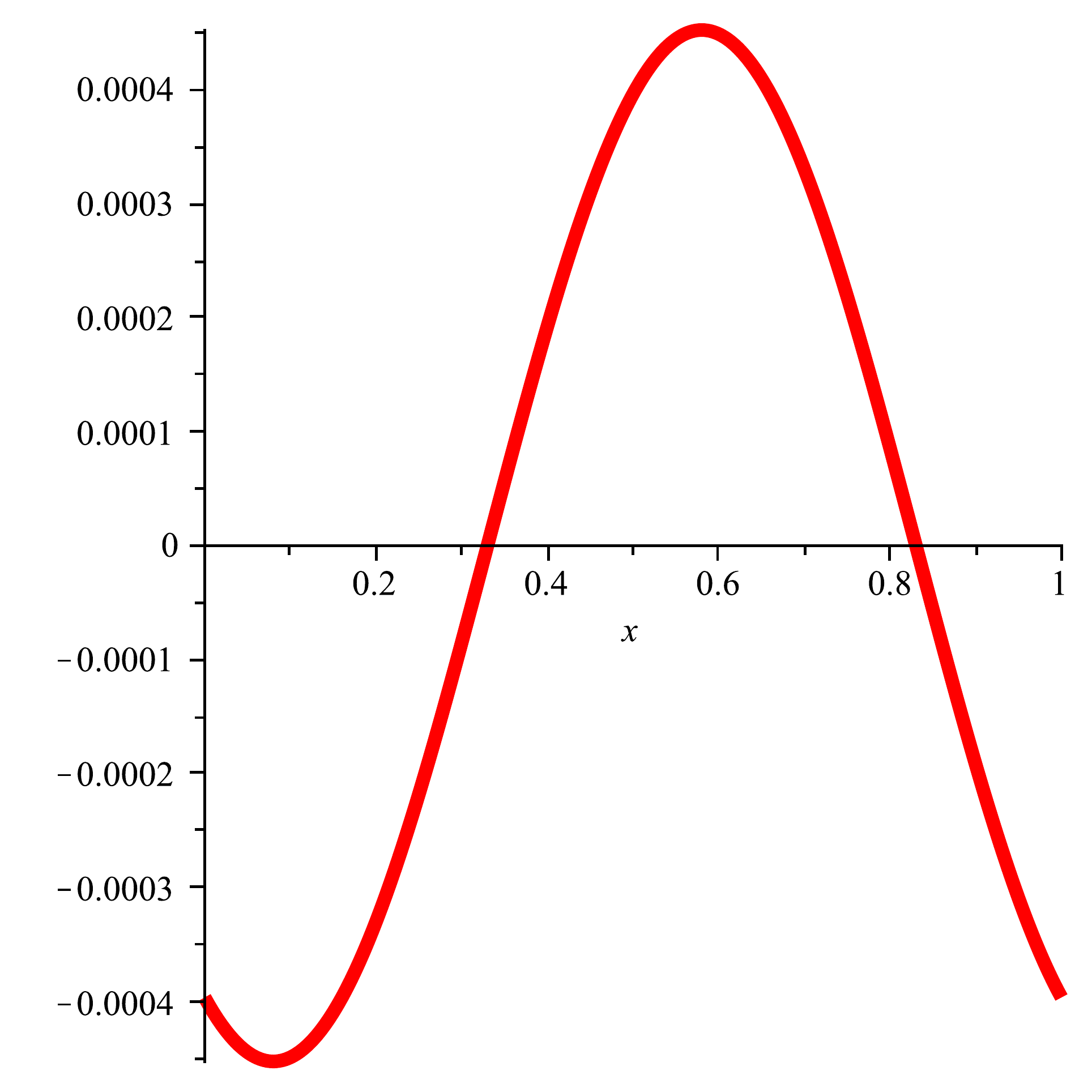}
\end{center}
\caption{\em $p=1/3$: Fluctuation of the periodic function
$Q_V(\log_3 n)$, as approximated by $\mathbb{V}(X_n)
+c_0/n-c_0/(2n^2)$ in logarithmic scale (left) and the first four
oscillating terms of its Fourier series (right). Here the number
$c_0=1/\log(1/p)^2$ and the two additional terms $c_0/n-c_0/(2n^2)$
are chosen for a better numerical correction and graphical display.}
\label{C:fig2}
\end{figure}

For computational purposes, we use the series expression
\begin{align*}
    \phi_V^*(\chi_k)
    &= \sum_{j\ge1}\frac{\mathbb{E}(X_j^2)}{j!}\,
    \Gamma(j+\chi_k)\left(q^j - 2^{-j-\chi_k}\right) +
    \Gamma(\chi_k) \left(1-2^{-\chi_k}\right)\\
    &\qquad +\sum_{j\ge1}\frac{\mu_j^{[2]}}{j!}
    \,\Gamma(j+\chi_k)\left(2\cdot 3^{-j-\chi_k}-4^{-j-\chi_k}
    -q^j2^{-j-\chi_k} \right)\\
    &\qquad +2\sum_{j\ge1}\frac{\mu_j}{j!}
    \,\Gamma(j+\chi_k)\left(2^{-j-\chi_k}-3^{-j-\chi_k}
    -q^j +q^j (1+p)^{-j-\chi_k}\right)\\
    &\qquad -2 \sum_{j\ge1}\frac{\mu_j^{[11]}}{j!}\,
    \Gamma(j+\chi_k)\left((1+p)^{-j-\chi_k}
    - (1+2p)^{-j-\chi_k}\right),
\end{align*}
where $\mu_n^{[2]} := n![z^n]f_1(z)^2$ and $\mu_n^{[11]} := n![z^n]
f_1(pz)f_1(qz)$.

\paragraph{Asymptotic distribution.} We now show that the
distribution of $X_n$ is asymptotically fluctuating and no
convergence to a fixed limit law is possible. We focus on deriving
an asymptotic approximation to the probability $\mathbb{P}(X_n=k)$,
which then leads to an effective estimate for the corresponding
distribution functions. The method of proof we use here relies on
the same analytic de-Poissonization procedure we used for the first
two moments, and requires a uniform estimate with respect to $k$;
see \cite{JS97,KPS96} for a similar analysis.

We begin by considering
\[
    \tilde{A}_k(z)=e^{-z}\sum_{n\ge 0}
    \mathbb{P}(X_n=k)\frac{z^n}{n!},
\]
which satisfies the obvious bound $\tilde{A}_k(x) \le 1$ for
real $x\ge0$. On the other hand, from \eqref{bivar-gf}, it follows
that
\begin{align}
    \tilde{A}_0(z)&=e^{-z}\nonumber\\
    \tilde{A}_{k+1}(z)&=\left(1-e^{-pz}\right)\tilde{A}_k(pz)
    +e^{-pz}\tilde{A}_{k+1}(qz)\quad (k\ge 0).\label{it-Ak}
\end{align}

Iterating \eqref{it-Ak} gives
\[
    \tilde{A}_{k+1}(z) = \sum_{j\ge0}e^{-(1-q^j)z}
    \left(1-e^{-pq^jz}\right)\tilde{A}_{k}(pq^jz),
    \qquad  \mbox{for \ } k\ge 0.
\]
We then deduce the explicit expressions
\begin{align}
    \tilde{A}_k(z)
    &= \sum_{j_1,\dots,j_k\ge0}
    e^{-(1- q\sum_{1\le r\le k}p^{r-1}q^{j_1+\cdots+j_r})z}
    \prod_{1\le r\le k}\left(1
    - e^{-p^rq^{j_1+\cdots+j_{r}}z}\right),\label{exp-Ak}
\end{align}
for $k\ge 1$.

Now define the normalizing function
\[
    \Omega(z):=\prod_{j\ge 0}\left(1-e^{-p^{-j}z}\right).
\]
For convenience, let
\[
    \eta(n) := \{\log_{1/p}n\}
\]
denote the fractional part of $\log_{1/p}n$.
\begin{thm}
The distribution of $X_n$ satisfies
\begin{align} \label{Xn-as-den}
    \mathbb{P}(X_n=\lfloor\log_{1/p}n\rfloor+k)
    =\sum_{j\ge0}\hat{R}_j\left(p^{-\eta(n)+k-j}\right)
    +O\left(\frac{1}{n}\right),
\end{align}
uniformly in $k\in\mathbb{Z}$, where $\hat{R}_k(z) :=
\Omega(z)e^{-pz} \tilde{A}_{k}(qz)$ and $\tilde{A}_k(z)$
is given in \eqref{exp-Ak}.
\end{thm}
\begin{proof} Write $\hat{A}_{k}(z):=\Omega(z)\tilde{A}_{k}(z)$.
Then by \eqref{it-Ak}, we have
\[
    \hat{A}_{k+1}(z)=\hat{A}_k(pz)+\hat{R}_{k+1}(z),
    \qquad \mbox{for \ } k\ge 0,
\]
which, after iteration, leads to
\[
    \hat{A}_k(z)=\sum_{0\le j\le k}\hat{R}_j(p^{k-j}z),
    \qquad  \mbox{for \ } k\ge 0,
\]
or, equivalently,
\[
    \tilde{A}_k(z)=\frac{1}{\Omega(z)}
    \sum_{0\le j\le k}\hat{R}_j(p^{k-j}z),
    \qquad \mbox{for \ }  k\ge 0.
\]
Since, by definition
\[
    \mathbb{P}(X_n=k) = n![z^n] e^z \tilde{A}_k(z),
\]
we need the following uniform estimates (which are needed to
justify the de-Poissonization; see Appendix).

\begin{lmm} The functions $\tilde{A}_k(z)$ are uniformly
JS-admissible, namely, for $\vert\arg(z)\vert\le\ve, 0<\ve<\pi/2$,
\begin{equation}\label{I}
    \tilde{A}_{k}(z)=O\left(\vert z\vert^{\ve'}\right),
\end{equation}
uniformly in $z$ and $k\ge0$, and, for
$\ve\le\vert\arg(z)\vert\le\pi$,
\begin{equation}\label{O}
    e^{z}\tilde{A}_{k}(z)=O
    \left(e^{(1-\ve')\vert z\vert}\right),
\end{equation}
uniformly in $z$ and $k\ge0$. Here $0<\ve'<1$ and the involved
constants in both cases are absolute.
\end{lmm}
\begin{proof}
Consider first $\vert\arg(z)\vert\le\ve$. Choose $K>0$ large enough
such that $1+2e^{-p\Re(z)}\le 1+\ve'$ for all $z$ with $\vert
z\vert>K$. Moreover, choose $C>0$ such that for $1\le\vert z\vert\le
K$
\[
    \vert\tilde{A}_k(z)\vert\le e^{\vert z\vert-\Re(z)}
    \le C\qquad \text{for}\ k\ge 0.
\]
We use a simple induction to show that
\begin{align} \label{Ak-bd}
    \vert\tilde{A}_k(z)\vert\le C\vert
    z\vert^{\log_{1/q}(1+\ve')}\qquad \text{for}\ k\ge 0.
\end{align}
A similar inductive proof is used in \cite{JS98} where it is
referred to as induction over increasing domains. The claim
\eqref{Ak-bd} holds for $k=0$. Next we assume \eqref{Ak-bd} has been
proved for $k$ and we prove it for $k+1$. The case $1\le\vert
z\vert\le K$ follows from the definition of $C$. If $K<\vert
z\vert\le K/q$, we can use \eqref{it-Ak} and the induction
hypothesis, and obtain
\begin{align*}
    \vert\tilde{A}_{k+1}(z)\vert
    &\le\left(1+e^{-p\Re(z)}\right)
    \vert\tilde{A}_k(pz)\vert+
    e^{-p\Re(z)}\vert\tilde{A}_{k+1}(qz)\vert\\
    &\le C(1+\ve')\vert qz\vert^{\log_{1/q}(1+\ve')}
    =C\vert z\vert^{\log_{1/q}(1+\ve')}.
\end{align*}
Continuing successively the same argument with $K/q^j<\vert
z\vert\le K/q^{j+1}$ for $j\ge1$, the upper bound \eqref{Ak-bd}
follows for all $z$. This concludes the proof of \eqref{I}.

To prove \eqref{O}, let $A_k(z):=e^{z}\tilde{A}_k(z)$. Then
\eqref{it-Ak} becomes
\[
    A_{k+1}(z)=\left(e^{qz}-e^{(q-p)z}\right)A_{k}(pz)
    +A_{k+1}(qz)\qquad (k\ge 0).
\]
Note that we have the (trivial) bound $\vert A_k(z)\vert\le e^{\vert
z\vert}$. Plugging this into the functional equation above yields
\[
    \vert A_{k+1}(z)\vert
    \le\left(e^{q\cos\ve\vert z\vert}
    +e^{(q-p)\cos(\ve)\vert z\vert}\right)
    e^{p\vert z\vert}+e^{q\vert z\vert},
\]
from which \eqref{O} follows.
\end{proof}

By a standard de-Poissonization argument (see Appendix for details
and references), we obtain
\begin{align}\label{asym-prob}
    \mathbb{P}(X_n= k)=\frac{1}{\Omega(n)}
    \sum_{0\le j\le k}\hat{R}_j(p^{k-j}n)
    +O\left(\frac{1}{n^{1-\ve}}\right),
\end{align}
uniformly in $k$, where $\ve>0$ is an arbitrary small constant.

Note that we have the identity
\begin{align} \label{id-density}
    \frac{1}{\Omega(n)}
    \sum_{k\ge0}\sum_{0\le j\le k}\hat{R}_j(p^{k-j}n) = 1.
\end{align}
This is seen as follows.
\begin{align*}
    \sum_{k\ge0}\sum_{0\le j\le k}\hat{R}_j(p^{k-j}n)
    &= \sum_{j\ge0}\sum_{k\ge0} \hat{R}_j(p^kn)\\
    &= \sum_{k\ge0} \Omega(p^kn)e^{-p^{k+1}n}\sum_{j\ge0}
    \tilde{A}_j(qp^kn) \\
    &= \sum_{k\ge0} \Omega(p^kn)
    \left(1-\left(1-e^{-p^{k+1}n} \right)\right)\\
    &= \sum_{k\ge0}\left(\Omega(p^kn)-\Omega(p^{k+1}n)\right)\\
    &= \Omega(n),
\end{align*}
which proves \eqref{id-density}.

Now
\begin{align*}
    \Omega(n)=\prod_{j\ge 0}\left(1-e^{-p^{-j}n}\right)
    =1+O\left(e^{-n}\right).
\end{align*}
This and \eqref{asym-prob} implies that
\[
    \sum_{0\le j\le k}\hat{R}_j(p^{k-j}n)=O(1),
\]
uniformly in $k$. Thus
\[
    \mathbb{P}(X_n=k)=\sum_{0\le j\le k}\hat{R}_j(p^{k-j}n)
    +O\left(\frac{1}{n^{1-\ve}}\right),
\]
uniformly in $k$.

Finally, observe that
\[
    \sum_{j\ge k+1}\hat{R}_j(p^{k-j}n)
    \le\sum_{j\ge k+1}e^{-p^{k+1-j}n}=
    \sum_{j\ge0}e^{-p^{-j}n}=O\left(e^{-n}\right).
\]
Thus
\begin{equation}\label{asym-pp}
    \mathbb{P}(X_n=k)=\sum_{j\ge0}
    \hat{R}_j(p^{k-j}n)
    +O\left(\frac{1}{n^{1-\ve}}\right),
\end{equation}
uniformly in $k$.

Since the mean is asymptotic to $\log_{1/p}n$, we replace $k$ by
$\lfloor\log_{1/p}n\rfloor+k$. Then
\[
    \mathbb{P}(X_n=\lfloor\log_{1/p}n\rfloor+k)
    =\sum_{j\ge0}\hat{R}_j(p^{-\eta(n)+k-j})
    +O\left(\frac{1}{n^{1-\ve}}\right),
\]
uniformly in $k$, where $\eta(n)=\{\log_{1/p}n\}$. Because of the
periodicity, the limiting distribution of $X_n-\lfloor \log_{1/p}n
\rfloor$, in general, does not exist. However, if we consider
instead a subsequence $n_j$ of positive integers such that
$\eta(n_j)\to \theta\in(0,1)$, as $j\to\infty$, then the limit law
does exist. The series on the right-hand side sums (over all $k$)
asymptotically to $1$ by \eqref{id-density}.

Finally, the finer error term $O(n^{-1})$ in \eqref{Xn-as-den} is
obtained by refining the same procedure by including, say, one more
term in the asymptotic expansion.
\end{proof}

On the other hand, from \eqref{exp-Ak}, we see that
$\mathbb{P}(X_{n}=k)$ is exponentially small for $k=O(1)$.

A similar analysis can be given for the distribution function of
$X_n$ (one only has to divide \eqref{bivar-gf} by $1-e^{y}$). This
then yields the following estimate for the distribution.

\begin{cor}
The distribution function of $X_n$ satisfies
\[
    \mathbb{P}(X_n-\lfloor\log_{1/p}n\rfloor
    \le k)=\sum_{j\ge0}\hat{S}_j(p^{-\eta(n)+k-j})
    +O\left(\frac{1}{n}\right),
\]
uniformly in $k\in\mathbb{Z}$, where $\hat{S}_k(z)=\sum_{j\le k}\hat{R}_j(z)$.
\end{cor}

When $p=1/2$, we have the representation
\[
    \sum_{k\ge0} \tilde{A}_k(z) u^k
    = \prod_{j\ge1}\left(1+(u-1)(1-e^{-z/2^j})\right),
\]
which gives
\[
    \tilde{A}_k(z)=e^{-z}\sum_{1\le j_1<\cdots<j_k}
    \prod_{1\le r\le k}\left(e^{z/2^{j_r}}-1\right);
\]
compare with \eqref{it-Ak}. This expression was already derived in
\cite{RJS93} where different expressions of the asymptotic
distributions are given.

\section*{Acknowledgements}

We thank the referee for helpful comments.

\section*{Appendix. Analytic de-Poissonization and JS-admissible
functions}

We develop the required tools for justifying the growth order of the
functions involved in this paper, as well as systematic means of
justifying the de-Poissonization procedure, based on the notion of
JS-admissible functions (combining ideas from Jacquet and Szpankowski
\cite{JS98} and the classical paper by Hayman \cite{Hayman56}).
The following materials, different from those in \cite{JS98}, are
modified from \cite{FHV13}, where more details are provided.

\begin{definition} An entire function $\tilde{f}$ is said to be
JS-admissible, denoted by $\tilde{f}\in\JS$, if the following two
conditions hold for $|z|\ge1$.
\begin{itemize}

\item[\textbf{(I)}] There exist $\alpha,\beta\in\mathbb{R}$ such
that uniformly for $|\arg(z)|\le\ve$,
\[
    \tilde{f}(z)
    = O\left(|z|^{\alpha} (\log_+|z|)^\beta\right),
\]
where $\log_+x := \log (1+x)$.
\item[\textbf{(O)}] Uniformly for $\ve\le|\arg(z)|\le\pi$,
\[
    f(z)
    := e^{z}\tilde{f}(z)
    = O\left(e^{(1-\ve')|z|}\right).
\]
\end{itemize}
Here and throughout this paper, the generic symbols $\ve, \ve'$
denote small quantities whose values are immaterial and not
necessarily the same at each occurrence.
\end{definition}
For convenience, we also write $\tilde{f}\in\JS_{\!\!\alpha,\beta}$
to indicate the growth order of $\tilde{f}$ inside the sector
$|\arg(z)|\le\ve$.

Note that if $\tilde{f}$ satisfies condition \textbf{(I)}, then, by
Cauchy's integral representation for derivatives (or by Ritt's
theorem; see \cite[Ch.~1, \S~4.3]{Olver74}), we have,
\begin{align*}
    \tilde{f}^{(k)}(z)
    &= O\left(|z|^{\alpha-k}(\log_+|z|)^\beta\right).
\end{align*}

On the other hand, by Cauchy's integral representation, we also have
\begin{align*}
    a_n
    &= \frac{n!}{2\pi i} \oint_{|z|=n}
    z^{-n-1} e^z\tilde{f}(z) \dd z \\
    &\approx \tilde{f}(n)
    \frac{n!}{2\pi i} \oint_{|z|=n} z^{-n-1} e^z \dd z\\
    &= \tilde{f}(n),
\end{align*}
since the saddle-point $z=n$ of the factor $z^{-n}e^z$ is unaltered
by the comparatively more smooth function $\tilde{f}(z)$.

The latter analytic viewpoint provides an additional advantage of
obtaining an expansion by using the Taylor expansion of $\tilde{f}$
at $z=n$, yielding
\begin{align}\label{PC}
    a_n
    =\sum_{j\ge0} \frac{\tilde{f}^{(j)}(n)}{j!}\tau_j(n),
\end{align}
where
\begin{align*}
    \tau_j(n)
    := n![z^n] (z-n)^j e^z
    = \sum_{0\le \ell \le j}
    \binom{j}{\ell} (-1)^{j-\ell} \frac{n!n^{j-\ell}}{(n-\ell)!}
    \qquad(j=0,1,\dots),
\end{align*}
and $[z^n]\phi(z)$ denotes the coefficient of $z^n$ in the Taylor
expansion of $\phi(z)$. We call such an expansion \emph{the
Poisson-Charlier expansion} since the $\tau_j$'s are essentially the
Charlier polynomials $C_j(\lambda,n)$ defined by
\[
    C_j(\lambda,n)
    := \lambda^{-n}n![z^n](z-1)^je^{\lambda z},
\]
so that $\tau_j(n) = n^j C_j(n,n)$. For other terms used in the
literature and more properties, see \cite{HFV10} and the references
therein. In particular, the expansion \eqref{PC} is absolutely
convergent when $\tilde{f}$ is entire.

\begin{prop}\label{prop-PC-asymp}
Assume $\tilde{f}\in\JS_{\!\!\alpha,\beta}$. Let $f(z) := e^z
\tilde{f}(z)$. Then the Poisson-Charlier expansion \eqref{PC} of
$f^{(n)}(0)$ is also an asymptotic expansion in the sense that
\begin{align*}
    a_n
    &:= f^{(n)}(0)
    = n![z^n]f(z)
    = n![z^n]e^z \tilde{f}(z) \\
    &= \sum_{0\le j<2k}
    \frac{\tilde{f}^{(j)}(n)}{j!}\,\tau_j(n) + O\left(n^{\alpha-k}
    \left(\log n\right)^\beta\right),
\end{align*}
for $k=1,2,\dots$.
\end{prop}

The polynomial growth of condition \textbf{(I)} is sufficient for
all our uses; see \cite{JS98} for more general versions.

The real advantage of introducing admissibility is that it opens the
possibility of developing closure properties as we now briefly discuss.

\begin{lmm} \label{lm-closure} Let $m$ be a nonnegative integer and
$\alpha\in(0,1)$.
\begin{itemize}

\item[(i)] $z^m, e^{-\alpha z}\in\JS$.

\item[(ii)] If $\tilde{f}\in\JS$, then $\tilde{f}(\alpha
z),z^m\tilde{f}\in\JS$.

\item[(iii)] If $\tilde{f}, \tilde{g} \in\JS$,
then $\tilde{f}+\tilde{g}\in\JS$.

\item[(iv)] If $\tilde{f}\in\JS$, then the product
$\tilde{P}\tilde{f}\in\JS$, where $\tilde{P}$ is a polynomial
of $z$.

\item[(v)] If $\tilde{f}, \tilde{g}\in\JS$, then $\tilde{h} \in\JS$,
where $\tilde{h}(z) := \tilde{f}(\alpha z)\tilde{g}((1-\alpha)z)$.

\item[(vi)] If $\tilde{f}\in\JS$, then $\tilde{f}'\in\JS$, and thus
$\tilde{f}^{(m)}\in\JS$.

\end{itemize}
\end{lmm}
\begin{proof} Straightforward and omitted.
\end{proof}

Specific to our need are the following transfer principles,
first the real version and then the complex one.
\begin{lmm} Let $\tilde{f}(z)$ and $\tilde{g}(z)$
be entire functions satisfying \label{lmm-fg}
\begin{align}\label{dfe-DST}
    \tilde{f}(z)=(1-e^{-pz})\tilde{f}(pz)
    +e^{-pz}\tilde{f}(qz) +\tilde{g}(z),
\end{align}
with $\tf(0)=\tilde{g}(0)=0$.
If $\tilde{g}(x) = O(x^{\alpha}(\log_+x)^\beta)$
for real large $x$, where $\alpha,\beta\in\mathbb{R}$, then
\begin{align}\label{trsfr}
    \tf(x) = \left\{\begin{array}{ll}
        O(x^{\alpha}(\log_+x)^\beta), & \text{if }\alpha>0;\\
        \left\{\begin{array}{ll}
            O((\log_+x)^{\beta+1}),& \text{if }\beta>-1\\
            O(\log_+ \log_+ x), & \text{if } \beta=-1\\
            O(1), &\text{if }\beta<-1
        \end{array}\right\}, & \text{if }\alpha=0;\\
        O(1), &\text{if } \alpha<0.
    \end{array}\right.
\end{align}
\end{lmm}
\begin{proof}
The idea of the proof here is that $\tf(x)$ behaves asymptotically like
the following recurrence
\[
    \phi(x) = \phi(px) + \tilde{g}(x),
\]
with $\phi(0)=0$. To that purpose, we need only to show that
$\tilde{f}(x)$ grows at most polynomially for large $x$. This is
easily achieved by noticing that $f$ is bounded above by the
function defined by the \emph{trie-recurrence}
\[
    \lambda(x) = \lambda(px)+\lambda(qx) +\nu(x),
\]
with $\lambda(0)=\nu(0)=0$, where
\[
    \nu(x) := \left\{\begin{array}{ll}
        Kx^{\bar{\alpha}}, &\text{if }x>1;\\
        K x, & \text{if }0\le x\le 1,
    \end{array}\right.
\]
$K>0$ being a large constant and $\bar{\alpha}:= \max\{\tr{\alpha},
0\}+1$. Note that the exact solution of $\lambda$ is given by
\[
    \lambda(x) = \sum_{j,\ell\ge0}\binom{j+\ell}{j}
    \nu\left(p^j q^\ell x\right).
\]
We then deduce, from this, that $\lambda(x) = O(x^{\bar{\alpha}})$
for large $x$. Accordingly, $\tilde{f}$ is polynomially bounded.
The more precise estimates \eqref{trsfr} then follows from
standard Mellin arguments (by subtracting the
first few $\bar{\alpha}+1$ terms of the Taylor expansion of
$\lambda(x)$ and then considering the Mellin transform of $\lambda$
so truncated, which exits in the strip $-\bar{\alpha}-1
<\Re(s)<-\bar{\alpha}$).
\end{proof}

\begin{prop} \label{prop-tr}
Let $\tilde{f}(z)$ and $\tilde{g}(z)$ be entire functions satisfying
\eqref{dfe-DST}. Then
\[
    \tilde{f}\in\JS
    \quad\text{if and only if}\quad\tilde{g}\in\JS.
\]
\end{prop}
\begin{proof}
The necessity part follows from Lemma~\ref{lm-closure}. We prove the
sufficiency, namely, if $\tilde{g}\in\JS$, then $\tilde{f}\in\JS$.

Write throughout the proof $z=re^{i\theta}$, $r\ge0$ and $-\pi \le
\theta\le \pi$. Consider first the region when $\ve \le |\theta|\le
\pi$. By assumption, $|e^z\tilde{g}(z)|\le K e^{(1-\ve_1)r}$. Define
\[
    M(r) := \max_{\ve \le|\theta|\le \pi} |f(z)| \qquad(r\ge0).
\]
Then by the functional equation
\[
    f(z) = \bigl(e^{qz}- e^{(q-p)z}\bigr)f(pz)
    + f(qz) +e^z\tilde{g}(z),
\]
we have
\[
    M(r) \le \bigl|e^{qz}-e^{(q-p)z}\bigr|M(pr)
    + M(qr) +Ke^{(1-\ve_1)r}.
\]
By using Pittel's inequality (see \cite[Appendix]{Pittel86})
\[
    |e^z-1| \le (e^r-1)e^{-r(1-\cos\theta)/2}
    \qquad(r\ge0; |\theta|\le\pi),
\]
we have
\begin{align}
    \bigl|e^{qz}-e^{(q-p)z}\bigr|
    &= \bigl|e^{(q-p)z}\bigr| \left|e^{pz}-1\right|\nonumber\\
    &\le e^{(q-p)r\cos\theta}\left(e^{pr}-1\right)
    e^{-pr(1-\cos\theta)/2}\nonumber\\
    &= \bigl(e^{qr}-e^{(q-p)r}\bigr)
    e^{-(q-p/2)r(1-\cos\theta)}\nonumber \\
    &\le e^{-\ve_2r}\bigl(e^{qr}-e^{(q-p)r}\bigr),\label{ineq-diff}
\end{align}
for $\ve\le|\theta|\le \pi$. Let $\ve' := \min\{\ve_1,\ve_2\}$. It
follows that
\[
    M(r) \le e^{-\ve'r} \bigl(e^{qr}-e^{(q-p)r}\bigr)M(pr)
    + M(qr) +Ke^{(1-\ve')r}.
\]
Let $\tilde{M}(r) := M(r) e^{-(1-\ve')r}$. Then
\[
    \tilde{M}(r) \le e^{-\ve'pr}\bigl(1-e^{-pr}\bigr)
    \tilde{M}(pr) + e^{-(1-\ve')pr}\tilde{M}(qr) + K.
\]
By the same bounding argument used in Lemma~\ref{lmm-fg}, we see
that $\tilde{M}(r) = O(1)$, and thus $M(r) = O(e^{(1-\ve')r})$.
[Technically, we define a function, say $\phi(r)$, satisfying the
functional equation
\[
    \phi(r) = e^{-\ve'pr}\bigl(1-e^{-pr}\bigr)
    \phi(pr) + e^{-(1-\ve')r} \phi(qr) + K,
\]
prove $\phi(r)=O(1)$ and then $M(r) \le \phi(r)$.]

We now consider the sector $|\theta|\le\ve$. Since $\tilde{g}(z)
= O\left(|z|^{\alpha} (\log_+|z|)^\beta\right)$ in this sector,
we can then show that $\tilde{f}$ grows at most polynomially
and is thus JS-admissible, details being omitted here.
\end{proof}

\end{document}